\documentclass[11pt,a4]{amsart}

\usepackage[mathscr]{eucal}
\usepackage{amsmath}
\usepackage{graphicx}
\usepackage{amssymb}
\usepackage{latexsym}
\usepackage{amsthm}
\usepackage{amscd}
\usepackage{enumerate}

\theoremstyle{plain}
\newtheorem{thm}{Theorem}

\theoremstyle{definition}

\newtheorem{ex}[thm]{Example}

\newcommand{\Z}{{\mathbb{Z}}}

\title{Proposal of a generating function of partition sequences
}
\author{MASANORI ANDO\ \ (Nara Gakuen University)}
\date{}
\begin{document}
\pagestyle{empty}
\maketitle\thispagestyle{empty}
\section{Introduction}
In this paper, we introduce the generating functions of partition sequences. 
Partition sequences have a one-to-one correspondence with partitions. 
Therefore, the generating function has no multiplicity and appears meaningless initially. 
However, we show that using a matrix can give meaning to the coefficients and preserve valuable information about partitions.
We also introduce some restrictions on partitions suitable for these generating functions.
\section{Integer partitions and partition sequences}
Let $n$ be a positive integer. A partition $\lambda $ of $n$ is an integer sequence 
\[
\lambda =(\lambda_1,\lambda_2,\ldots,\lambda_\ell),
\]
satisfying $\lambda _1 \geq \lambda _2 \geq \ldots \geq \lambda _\ell >0$ and $\displaystyle \sum _{i=1}^{\ell }{\lambda _i} =n$. 
We call $\ell (\lambda):= \ell$ the length of $\lambda $, 
$|\lambda|:=n$ the size of $\lambda $, 
and each $\lambda _i$ a part of $\lambda $. 
We let ${\mathcal {P}} $ and ${\mathcal {P}}(n) $ denote the set of partitions and that of $n$. 
And let ${\mathcal {OP}} $ the set of partitions that all parts are odd, 
${\mathcal {SP}} $ the set of partitions that does not have the same size parts. 
These are the most major restrictions on the set of partitions. 
For a partition $\lambda$, the Young diagram of $\lambda$ is defined by
\[
Y(\lambda):=
\{
(i,j)\in \Z^2\ |\ 1\leq i\leq \ell(\lambda), 1\leq j \leq \lambda_i
\}
\]

\begin{ex}
For $\lambda=(5,2,2)\in {\mathcal{P}}(9)$,

\unitlength 0.1in
\begin{picture}(20.00,6.00)(4.00,-10.00)
\put(5.0000,-8.0000){\makebox(0,0)[lb]{$Y(5,2,2)=$}}%
\put(26,-10){\makebox(0,0)[lb]{.}}
%
\special{pn 8}%
\special{pa 1400 400}%
\special{pa 1400 1000}%
\special{fp}%
\special{pa 1400 1000}%
\special{pa 1800 1000}%
\special{fp}%
\special{pa 1400 800}%
\special{pa 1800 800}%
\special{fp}%
\special{pa 1800 1000}%
\special{pa 1800 600}%
\special{fp}%
\special{pa 1400 600}%
\special{pa 2400 600}%
\special{fp}%
\special{pa 2400 600}%
\special{pa 2400 400}%
\special{fp}%
\special{pa 2400 400}%
\special{pa 1400 400}%
\special{fp}%
\special{pa 1800 400}%
\special{pa 1800 600}%
\special{fp}%
\special{pa 2000 600}%
\special{pa 2000 400}%
\special{fp}%
\special{pa 2200 400}%
\special{pa 2200 600}%
\special{fp}%
\special{pa 1600 400}%
\special{pa 1600 1000}%
\special{fp}%
\end{picture}%
\end{ex}

And we define 
\begin{eqnarray*}
H_{i,j}(\lambda)&:=&\{
(a,b)\in Y(\lambda)\ |\ ( a=i \wedge b\geq j) \vee (a\geq i \wedge b=j)
\}
, \\
h_{i,j}(\lambda)&:=&\sharp H_{i,j}(\lambda)
\end{eqnarray*}
We call $H_{i,j}(\lambda)$ the $(i,j)$-hook of  $\lambda$ and $h_{i,j}(\lambda)$ the $(i,j)$-hook length of  $\lambda$. 
Especially, 
we call the $(i,j)$-hook satisfying $(i,j+1)\not\in Y(\lambda)$ a vertical hook. 
And we call the $(i,j)$-hook satisfying $(i+1,j)\not\in Y(\lambda)$ a horizontal hook.

\begin{ex}
The number in each cell is the hook length,

\unitlength 0.1in
\begin{picture}(20.00,6.00)(4.00,-10.00)
\put(5.0000,-8.0000){\makebox(0,0)[lb]{$Y(5,2,2)=$}}%
\put(26.0000,-10.0000){\makebox(0,0)[lb]{.}}%
%
\special{pn 8}%
\special{pa 1400 400}%
\special{pa 1400 1000}%
\special{fp}%
\special{pa 1400 1000}%
\special{pa 1800 1000}%
\special{fp}%
\special{pa 1400 800}%
\special{pa 1800 800}%
\special{fp}%
\special{pa 1800 1000}%
\special{pa 1800 600}%
\special{fp}%
\special{pa 1400 600}%
\special{pa 2400 600}%
\special{fp}%
\special{pa 2400 600}%
\special{pa 2400 400}%
\special{fp}%
\special{pa 2400 400}%
\special{pa 1400 400}%
\special{fp}%
\special{pa 1800 400}%
\special{pa 1800 600}%
\special{fp}%
\special{pa 2000 600}%
\special{pa 2000 400}%
\special{fp}%
\special{pa 2200 400}%
\special{pa 2200 600}%
\special{fp}%
\special{pa 1600 400}%
\special{pa 1600 1000}%
\special{fp}%
\put(14.5000,-5.5000){\makebox(0,0)[lb]{$7$}}%
\put(16.5000,-5.5000){\makebox(0,0)[lb]{$6$}}%
\put(18.5000,-5.5000){\makebox(0,0)[lb]{$3$}}%
\put(20.5000,-5.5000){\makebox(0,0)[lb]{$2$}}%
\put(22.5000,-5.5000){\makebox(0,0)[lb]{$1$}}%
\put(16.5000,-7.5000){\makebox(0,0)[lb]{$2$}}%
\put(16.5000,-9.5000){\makebox(0,0)[lb]{$1$}}%
\put(14.5000,-9.5000){\makebox(0,0)[lb]{$2$}}%
\put(14.500,-7.5000){\makebox(0,0)[lb]{$3$}}%
\end{picture}%
\end{ex}

We define $Y(\lambda \setminus H_{i,j}(\lambda))$ as follows. 
First, we remove the boxes corresponding to $H_{i,j}(\lambda)$ from $Y(\lambda)$. 
At that time, it will be divided into two diagrams, so slide the bottom right diagram to the top left.

\begin{ex}
For $\lambda=(5,4,4,2)$, $(i,j)=(2,2)$,

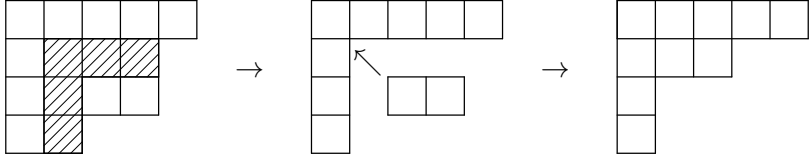
\unitlength 0.1in
\begin{picture}(42.00,8.00)(4.00,-12.00)

\special{pn 8}%
\special{pa 400 400}%
\special{pa 1400 400}%
\special{fp}%
\special{pa 1400 400}%
\special{pa 1400 600}%
\special{fp}%
\special{pa 1400 600}%
\special{pa 400 600}%
\special{fp}%
\special{pa 400 400}%
\special{pa 400 1200}%
\special{fp}%
\special{pa 400 1200}%
\special{pa 800 1200}%
\special{fp}%
\special{pa 800 1200}%
\special{pa 800 400}%
\special{fp}%
\special{pa 1200 400}%
\special{pa 1200 1000}%
\special{fp}%
\special{pa 1200 1000}%
\special{pa 400 1000}%
\special{fp}%
\special{pa 400 800}%
\special{pa 1200 800}%
\special{fp}%
\special{pa 1000 1000}%
\special{pa 1000 400}%
\special{fp}%
\special{pa 600 400}%
\special{pa 600 1200}%
\special{fp}%

\special{pn 8}%
\special{pa 600 1200}%
\special{pa 800 1200}%
\special{pa 800 800}%
\special{pa 1200 800}%
\special{pa 1200 600}%
\special{pa 600 600}%
\special{pa 600 600}%
\special{pa 600 1200}%
\special{fp}%

\special{pn 4}%
\special{pa 600 620}%
\special{pa 620 600}%
\special{fp}%
\special{pa 600 680}%
\special{pa 680 600}%
\special{fp}%
\special{pa 600 740}%
\special{pa 740 600}%
\special{fp}%
\special{pa 600 800}%
\special{pa 800 600}%
\special{fp}%
\special{pa 600 860}%
\special{pa 860 600}%
\special{fp}%
\special{pa 600 920}%
\special{pa 920 600}%
\special{fp}%
\special{pa 600 980}%
\special{pa 980 600}%
\special{fp}%
\special{pa 600 1040}%
\special{pa 800 840}%
\special{fp}%
\special{pa 840 800}%
\special{pa 1040 600}%
\special{fp}%
\special{pa 600 1100}%
\special{pa 800 900}%
\special{fp}%
\special{pa 900 800}%
\special{pa 1100 600}%
\special{fp}%
\special{pa 600 1160}%
\special{pa 800 960}%
\special{fp}%
\special{pa 960 800}%
\special{pa 1160 600}%
\special{fp}%
\special{pa 620 1200}%
\special{pa 800 1020}%
\special{fp}%
\special{pa 1020 800}%
\special{pa 1200 620}%
\special{fp}%
\special{pa 680 1200}%
\special{pa 800 1080}%
\special{fp}%
\special{pa 1080 800}%
\special{pa 1200 680}%
\special{fp}%
\special{pa 740 1200}%
\special{pa 800 1140}%
\special{fp}%
\special{pa 1140 800}%
\special{pa 1200 740}%
\special{fp}%

\put(16.0000,-8.0000){\makebox(0,0)[lb]{$\rightarrow$}}%

\special{pn 8}%
\special{pa 2000 400}%
\special{pa 2000 1200}%
\special{fp}%
\special{pa 2000 1200}%
\special{pa 2200 1200}%
\special{fp}%
\special{pa 2200 1200}%
\special{pa 2200 400}%
\special{fp}%
\special{pa 2000 400}%
\special{pa 3000 400}%
\special{fp}%
\special{pa 3000 400}%
\special{pa 3000 600}%
\special{fp}%
\special{pa 3000 600}%
\special{pa 2000 600}%
\special{fp}%
\special{pa 2000 800}%
\special{pa 2200 800}%
\special{fp}%
\special{pa 2200 1000}%
\special{pa 2000 1000}%
\special{fp}%
\special{pa 2400 600}%
\special{pa 2400 400}%
\special{fp}%
\special{pa 2600 400}%
\special{pa 2600 600}%
\special{fp}%
\special{pa 2800 600}%
\special{pa 2800 400}%
\special{fp}%
\special{pa 2400 800}%
\special{pa 2800 800}%
\special{fp}%
\special{pa 2800 800}%
\special{pa 2800 1000}%
\special{fp}%
\special{pa 2800 1000}%
\special{pa 2400 1000}%
\special{fp}%
\special{pa 2400 1000}%
\special{pa 2400 800}%
\special{fp}%
\special{pa 2600 800}%
\special{pa 2600 1000}%
\special{fp}%

\put(32.0000,-8.0000){\makebox(0,0)[lb]{$\rightarrow$}}%

\special{pn 8}%
\special{pa 3600 400}%
\special{pa 4600 400}%
\special{fp}%
\special{pa 4600 400}%
\special{pa 4600 600}%
\special{fp}%
\special{pa 4600 600}%
\special{pa 3600 600}%
\special{fp}%
\special{pa 3600 600}%
\special{pa 3600 400}%
\special{fp}%
\special{pa 3600 400}%
\special{pa 3600 1200}%
\special{fp}%
\special{pa 3600 1200}%
\special{pa 3800 1200}%
\special{fp}%
\special{pa 3800 1200}%
\special{pa 3800 400}%
\special{fp}%
\special{pa 3600 1000}%
\special{pa 3800 1000}%
\special{fp}%
\special{pa 3600 800}%
\special{pa 4200 800}%
\special{fp}%
\special{pa 4200 800}%
\special{pa 4200 400}%
\special{fp}%
\special{pa 4000 400}%
\special{pa 4000 800}%
\special{fp}%
\special{pa 4400 400}%
\special{pa 4400 600}%
\special{fp}%
\put(22.2000,-7.9000){\makebox(0,0)[lb]{$\nwarrow$}}%
\end{picture}%

Then, $(5,4,4,2)\setminus H_{2,2}(5,4,4,2)=(5,3,1,1)$. 
\end{ex}

For any partion $\lambda$  positive integer $p$, 
define $\lambda_{(p)}$ as the result of repeating the operation of 
removing a hook of length $p$ from $\lambda$ until there are no more hooks of length $p$.

\begin{ex}For $\lambda=(5,4,4,2)$, $p=3$,

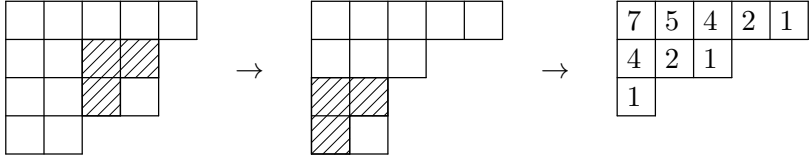
\unitlength 0.1in
\begin{picture}(42.00,8.00)(4.00,-12.00)

\special{pn 8}%
\special{pa 400 400}%
\special{pa 1400 400}%
\special{fp}%
\special{pa 1400 400}%
\special{pa 1400 600}%
\special{fp}%
\special{pa 1400 600}%
\special{pa 400 600}%
\special{fp}%
\special{pa 400 400}%
\special{pa 400 1200}%
\special{fp}%
\special{pa 400 1200}%
\special{pa 800 1200}%
\special{fp}%
\special{pa 800 1200}%
\special{pa 800 400}%
\special{fp}%
\special{pa 400 1000}%
\special{pa 1200 1000}%
\special{fp}%
\special{pa 1200 1000}%
\special{pa 1200 400}%
\special{fp}%
\special{pa 1000 400}%
\special{pa 1000 1000}%
\special{fp}%
\special{pa 1200 800}%
\special{pa 400 800}%
\special{fp}%
\special{pa 600 400}%
\special{pa 600 1200}%
\special{fp}%

\special{pn 8}%
\special{pa 800 1000}%
\special{pa 1000 1000}%
\special{pa 1000 800}%
\special{pa 1200 800}%
\special{pa 1200 600}%
\special{pa 800 600}%
\special{pa 800 600}%
\special{pa 800 1000}%
\special{fp}%

\special{pn 4}%
\special{pa 800 660}%
\special{pa 860 600}%
\special{fp}%
\special{pa 800 720}%
\special{pa 920 600}%
\special{fp}%
\special{pa 800 780}%
\special{pa 980 600}%
\special{fp}%
\special{pa 800 840}%
\special{pa 1040 600}%
\special{fp}%
\special{pa 800 900}%
\special{pa 1100 600}%
\special{fp}%
\special{pa 800 960}%
\special{pa 1160 600}%
\special{fp}%
\special{pa 820 1000}%
\special{pa 1000 820}%
\special{fp}%
\special{pa 1020 800}%
\special{pa 1200 620}%
\special{fp}%
\special{pa 880 1000}%
\special{pa 1000 880}%
\special{fp}%
\special{pa 1080 800}%
\special{pa 1200 680}%
\special{fp}%
\special{pa 940 1000}%
\special{pa 1000 940}%
\special{fp}%
\special{pa 1140 800}%
\special{pa 1200 740}%
\special{fp}%

\special{pn 8}%
\special{pa 2000 400}%
\special{pa 3000 400}%
\special{fp}%
\special{pa 3000 400}%
\special{pa 3000 600}%
\special{fp}%
\special{pa 3000 600}%
\special{pa 2000 600}%
\special{fp}%
\special{pa 2000 400}%
\special{pa 2000 1200}%
\special{fp}%
\special{pa 2000 1200}%
\special{pa 2400 1200}%
\special{fp}%
\special{pa 2400 1200}%
\special{pa 2400 400}%
\special{fp}%
\special{pa 2000 1000}%
\special{pa 2400 1000}%
\special{fp}%
\special{pa 2000 800}%
\special{pa 2600 800}%
\special{fp}%
\special{pa 2600 800}%
\special{pa 2600 400}%
\special{fp}%
\special{pa 2800 400}%
\special{pa 2800 600}%
\special{fp}%
\special{pa 2200 400}%
\special{pa 2200 1200}%
\special{fp}%

\special{pn 8}%
\special{pa 2000 1200}%
\special{pa 2200 1200}%
\special{pa 2200 1000}%
\special{pa 2400 1000}%
\special{pa 2400 800}%
\special{pa 2000 800}%
\special{pa 2000 800}%
\special{pa 2000 1200}%
\special{fp}%
 
\special{pn 4}%
\special{pa 2000 840}%
\special{pa 2040 800}%
\special{fp}%
\special{pa 2000 900}%
\special{pa 2100 800}%
\special{fp}%
\special{pa 2000 960}%
\special{pa 2160 800}%
\special{fp}%
\special{pa 2000 1020}%
\special{pa 2220 800}%
\special{fp}%
\special{pa 2000 1080}%
\special{pa 2280 800}%
\special{fp}%
\special{pa 2000 1140}%
\special{pa 2340 800}%
\special{fp}%
\special{pa 2000 1200}%
\special{pa 2200 1000}%
\special{fp}%
\special{pa 2200 1000}%
\special{pa 2400 800}%
\special{fp}%
\special{pa 2060 1200}%
\special{pa 2200 1060}%
\special{fp}%
\special{pa 2260 1000}%
\special{pa 2400 860}%
\special{fp}%
\special{pa 2120 1200}%
\special{pa 2200 1120}%
\special{fp}%
\special{pa 2320 1000}%
\special{pa 2400 920}%
\special{fp}%
\special{pa 2180 1200}%
\special{pa 2200 1180}%
\special{fp}%
\special{pa 2380 1000}%
\special{pa 2400 980}%
\special{fp}%

\put(32.0000,-8.0000){\makebox(0,0)[lb]{$\rightarrow$}}%

\put(16.0000,-8.0000){\makebox(0,0)[lb]{$\rightarrow$}}%

\special{pn 8}%
\special{pa 3600 400}%
\special{pa 4600 400}%
\special{fp}%
\special{pa 4600 400}%
\special{pa 4600 600}%
\special{fp}%
\special{pa 4600 600}%
\special{pa 3600 600}%
\special{fp}%
\special{pa 3600 800}%
\special{pa 4200 800}%
\special{fp}%
\special{pa 4200 400}%
\special{pa 4200 800}%
\special{fp}%
\special{pa 3800 400}%
\special{pa 3800 1000}%
\special{fp}%
\special{pa 3800 1000}%
\special{pa 3600 1000}%
\special{fp}%
\special{pa 3600 1000}%
\special{pa 3600 400}%
\special{fp}%
\special{pa 4000 400}%
\special{pa 4000 800}%
\special{fp}%
\special{pa 4400 400}%
\special{pa 4400 600}%
\special{fp}%

\put(44.5000,-5.5000){\makebox(0,0)[lb]{$1$}}%

\put(42.5000,-5.5000){\makebox(0,0)[lb]{$2$}}%

\put(40.5000,-5.5000){\makebox(0,0)[lb]{$4$}}%

\put(40.5000,-7.5000){\makebox(0,0)[lb]{$1$}}%

\put(38.5000,-7.5000){\makebox(0,0)[lb]{$2$}}%

\put(38.5000,-5.5000){\makebox(0,0)[lb]{$5$}}%

\put(36.5000,-5.5000){\makebox(0,0)[lb]{$7$}}%

\put(36.5000,-7.5000){\makebox(0,0)[lb]{$4$}}%

\put(36.5000,-9.5000){\makebox(0,0)[lb]{$1$}}%
\end{picture}%

Then, $(5,4,4,2)_{(3)}=(5,3,1)$. 
\end{ex}

When $\lambda=\lambda_{(p)}$, we call the partition $\lambda$ is $p$-core. 
i.e. 
\[
{}^\forall (i,j)\in Y(\lambda), h_{i,j}(\lambda)\not=p.
\]
Let $C_{(p)}$ be the set of $p$-core partitions.

\begin{ex}
\[
C_{(2)}=\{
(), (1), (2, 1), (3, 2, 1), (4, 3, 2, 1), \ldots
\}, 
\]
\[
C_{(3)}=\{
(), (1), (2), (1,1), (3, 1), (2, 1, 1), \ldots
\}. 
\]
\end{ex}
And we let $vhC_{(p)}$ the set of partitions that does not have vartical hook and horizontal hook of length $p$. 

For partition $\lambda$, 
partition sequence \cite{ol} $M(\lambda)$ 
is a sequence of $X$ and $Y$ of length $h_{1,1}(\lambda)+1$, 
with $X$ placed in the $(h_{k,1}+1)$-th position for all $k$. 
In \cite{ol}, the partition sequence is a double infinite sequence of $0$ and $1$. 
This is also called the Maya diagram, and is used for example in Sato-Welter game\cite{en}\cite{w}. 
In this paper, we cut the left side $00\ldots 0$ and the right side $11\ldots 1$. 
And replace 0 with $X$ and 1 with $Y$. 
$M(\lambda)$ equals the rim of $Y(\lambda)$. 
When viewed as a path from $(\ell(\lambda), 0)$ to $(0,\lambda_1)$, it is equivalent to replacing $\rightarrow$ with $Y$ and $\uparrow$ with $X$.

\begin{ex}
For $\lambda=(5,2,2)$, $M(\lambda)=YYXXYYYX$. 
\end{ex}
\section{Generating functions of partition sequences}
\begin{thm}
\[
\sum_{\lambda\in {\mathcal{P}}}{M(\lambda)}={1+Y\frac{1}{1-X-Y}X}
\]
\end{thm}
\begin{proof}
\[
\displaystyle 
\frac{1}{1-X-Y}=\sum_{k=0}^{\infty}(X+Y)^k, 
\]
which is the sum of all sequences of $X$ and $Y$. 
Subsequently, 
$\displaystyle Y\frac{1}{1-X-Y}X$ is the sum of all sequences of $X$ and $Y$ starting with Y and ending with X. 
They correspond to partitions of length greater than or equal to $1$. 
And $1$ corresponds to the partition with no parts. 
\end{proof}
This is the generating function of partition sequences. 
But all the coefficients are $1$, so it has little meaning as a generating function. 
However, by substituting an appropriate matrix, it is possible to give meaning to the coefficients. 
\begin{thm}
For a partition $\lambda$ and its partition sequence $M(\lambda)$, we put
\[
X=
\left[
\begin{array}{ccc}
1&0&0\\
0&1&1\\
0&0&1
\end{array}
\right]
, 
Y=
\left[
\begin{array}{ccc}
1&1&0\\
0&1&0\\
0&0&1
\end{array}
\right]
, 
\]
then, 
\[
M(\lambda)=
\left[
\begin{array}{ccc}
1&\lambda_1&|\lambda|\\
0&1&\ell(\lambda)\\
0&0&1
\end{array}
\right]
.
\]
\end{thm}
\begin{proof}
First, 
\[
1=E_3=\left[
\begin{array}{ccc}
1&0&0\\
0&1&0\\
0&0&1
\end{array}
\right]
=M().
\]
It is the intial value.
\[
YM(\lambda)=
\left[
\begin{array}{ccc}
1&\lambda_1+1&|\lambda|+\ell(\lambda)\\
0&1&\ell(\lambda)\\
0&0&1
\end{array}
\right]
, 
\]
which corresponds to adding $1$ to all parts of $\lambda$.
In the rim movement, 
we insert $\rightarrow$ first. 
\[
XM(\lambda)=
\left[
\begin{array}{ccc}
1&\lambda_1&|\lambda|\\
0&1&\ell(\lambda)+1\\
0&0&1
\end{array}
\right]
, 
\]
which corresponds to adding part $0$ to $\lambda$ as $\lambda_{\ell(\lambda)+1}$. 
In the rim movement, 
we insert $\uparrow$ first. 
\end{proof}
Herein, the product is calculated as a product of matrices, and the sum is formal. 
If you calculate the sum of the matrices, the components will be infinite, and the calculation will become meaningless. 
Reducing the amount of information, makes it possible to give meaning to the coefficients. 
For example, 
if you delete the second row and column, 
you can obtain the same value as the generating function of the number of partition size $n$. 
\[
\left[
\begin{array}{ccc}
1&0&0\\
0&0&0\\
0&0&1
\end{array}
\right]
\left(
\sum_{\lambda\in{\mathcal{P}}}{M(\lambda)}
\right)
\left[
\begin{array}{ccc}
1&0&0\\
0&0&0\\
0&0&1
\end{array}
\right]
=\sum_{n\geq 0}{\sharp{\mathcal{P}(n)
}\left[
\begin{array}{ccc}
1&0&n\\
0&0&0\\
0&0&1
\end{array}
\right]. 
}
\]
The advantage of using a generating function of a partition sequence is 
that it is easy to construct a generating function of some restrictions on the partitions. 
Moreover, the generating function simultaneously stores three pieces of information:
the first part, the length, and the size. 
We can expect further developments, such as finding another suitable matrix and generalizing it to three-dimensional partitions. 
We look at some other useful matrices. 
\begin{thm}
For a partition $\lambda$ and its partition sequence $M(\lambda)$, we put
\[
X=
\left[
\begin{array}{ccc}
1&0&0\\
0&1&-1\\
0&0&-1
\end{array}
\right]
, 
Y=
\left[
\begin{array}{ccc}
1&1&0\\
0&1&0\\
0&0&1
\end{array}
\right]
, 
\]
then, 
\[
M(\lambda)=
\left[
\begin{array}{ccc}
1&\lambda_1&-a(\lambda)\\
0&1&-\overline{\ell(\lambda)}\\
0&0&(-1)^{\ell(\lambda)}
\end{array}
\right]
.
\]
Here $a(\lambda)$ is alternating sum of $\lambda$. 
That is 
\[
a(\lambda):=\sum_{k=1}^{\ell(\lambda)}{(-1)^k\lambda_k}. 
\]
And $\overline{\ell(\lambda)}$ is the meaning in $\Z/2\Z$. 
\[
\overline{\ell(\lambda)}=\left\{
\begin{array}{lc}
1&(\ell(\lambda)\ \textrm{is odd})\\
0&(\ell(\lambda)\ \textrm{is even})
\end{array}
\right.
\]
\end{thm}
\begin{proof}
\[
YM(\lambda)=
\left[
\begin{array}{ccc}
1&\lambda_1+1&-a(\lambda)-\overline{\ell(\lambda)}\\
0&1&-\overline{\ell(\lambda)}\\
0&0&(-1)^{\ell(\lambda)}
\end{array}
\right]
, 
\]
which corresponds to adding $1$ to all parts of $\lambda$.
When $\ell(\lambda)$ is even, $a(\lambda)$ does not change becouse it is alternating sum of even parts. 
And when $\ell(\lambda)$ is odd, $a(\lambda)$ increase just $1$. 
\[
XM(\lambda)=
\left[
\begin{array}{ccc}
1&\lambda_1&-a(\lambda)\\
0&1&-\overline{\ell(\lambda)+1}\\
0&0&(-1)^{\ell(\lambda)+1}
\end{array}
\right]
, 
\]
which corresponds to adding part $0$ to $\lambda$ as $\lambda_{\ell(\lambda)+1}$. 
Then the new alternating sum is either $a(\lambda)+0$ or $a(\lambda)-0$, 
unchanged in either case. 
\end{proof}
$X$ and $Y$ are symmetric about conjugate. 
So when 
$
X=
\left[
\begin{array}{ccc}
1&0&0\\
0&1&1\\
0&0&1
\end{array}
\right]
, 
Y=
\left[
\begin{array}{ccc}
-1&-1&0\\
0&1&0\\
0&0&1
\end{array}
\right]
$
, 
$a(\lambda')$ appears in the $(1,3)$-element of $M(\lambda)$. 
It is a natural question what happens to $M(\lambda)$ with respect to the following conditions. 
\begin{thm}
For a partition $\lambda$ and its partition sequence $M(\lambda)$, we put
\[
X=
\left[
\begin{array}{ccc}
1&0&0\\
0&1&-1\\
0&0&-1
\end{array}
\right]
, 
Y=
\left[
\begin{array}{ccc}
-1&-1&0\\
0&1&0\\
0&0&1
\end{array}
\right]
, 
\]
then, 
\[
M(\lambda)=
\left[
\begin{array}{ccc}
(-1)^{\lambda_1}&-\overline{\lambda_1}&b(\lambda)\\
0&1&-\overline{\ell(\lambda)}\\
0&0&(-1)^{\ell(\lambda)}
\end{array}
\right]
.
\]
Here,  
\[
b(\lambda)=
\left\{
\begin{array}{lc}
n&(\lambda_{(2)}=(2n-1, \ldots, 2,1))\\
-n&(\lambda_{(2)}=(2n,2n-1, \ldots, 2,1))
\end{array}
\right.
\]
Therefore, the $(1,3)$-element of  $M(\lambda)$ determines $\lambda_{(2)}$. 
\end{thm}
\begin{proof}
Focus on $XX=YY=E_3$. 
\end{proof}
In the case of 
$
X=
\left[
\begin{array}{ccc}
1&0&0\\
0&-1&-1\\
0&0&1
\end{array}
\right]
, 
Y=
\left[
\begin{array}{ccc}
1&-1&0\\
0&-1&0\\
0&0&1
\end{array}
\right]
$
, also $XX=YY=E_3$. 
\begin{thm}
For a partition $\lambda$ and its partition sequence $M(\lambda)$, we put
\[
X=
\left[
\begin{array}{ccc}
1&0&0\\
0&-1&-1\\
0&0&1
\end{array}
\right]
, 
Y=
\left[
\begin{array}{ccc}
1&-1&0\\
0&-1&0\\
0&0&1
\end{array}
\right]
, 
\]
then, 
\begin{eqnarray*}
&&M(\lambda)\\
&=&
\left[
\begin{array}{ccc}
1&(-1)^{\lambda_1+\ell(\lambda_{(2)})}(-\ell(\lambda_{(2)})+\frac{1}{2})-\frac{1}{2}&|\lambda_{(2)}|\\
0&(-1)^{\lambda_1+\ell(\lambda)}&(-1)^{\ell(\lambda)+\ell(\lambda_{(2)})}(-\ell(\lambda_{(2)})+\frac{1}{2})-\frac{1}{2}\\
0&0&1
\end{array}
\right]
.
\end{eqnarray*}
In particular, $(1,3)$-element is always triangle number. 
We can calculate $\lambda_{(2)}$'s information more directly.
\end{thm}
When constructing a generating function of partition, we often want to create a form that has information about $\ell(\lambda)$ and $a(\lambda)$.
However, in order to do this, it is necessary to understand how the underlying generating function is constructed and to make appropriate modifications.
Also, just because you understand doesn't mean it's always possible.
However, with the partition sequence type generating function introduced in this paper, this is possible simply by substituting a fixed matrix. 
\section{Appendix,  some examples}
First, 
we construct the partition sequence type generating functions of $\mathcal{OP}$ and $\mathcal{SP}$. 
It's not that difficult if you focus on the rim movement. 
\[
\sum_{\lambda\in \mathcal{OP}}{M(\lambda)}=1+Y\frac{1}{1-X-YY}X
\]
Since $Y$ appears in a set of two except for the first $Y$, all parts are odd numbers. 
\[
\sum_{\lambda\in \mathcal{SP}}{M(\lambda)}=1+Y\frac{1}{1-XY-Y}X
\]
Because $X$ is not continuous, parts of the same size do not appear.

In Section 2, we defined the $p$-core partitions. 
It is well known,
\[
\sum_{\lambda\in {\mathcal{C}}_{(p)}}{x^{|\lambda|}}=
\prod_{k=1}^\infty{\frac{(1-x^{pk})^p}{1-x^k}}. 
\]
They are also suitable for the generation function of partition sequences. 
For $p=2$, 
the rim movement of $2$-core partitions is a continuation of $\rightarrow \uparrow$. 
Then, 
\[
\sum_{\lambda\in {\mathcal{C}}_{(2)}}{M(\lambda)}=
\frac{1}{1-YX}.
\]
For $p=3$, 
the rim of $3$-core partitions consists of the following three patterns: $\rightarrow\uparrow\uparrow$, $\rightarrow\uparrow$, and 
$\rightarrow\rightarrow\uparrow$. 
The connections between them are as follows. 
Anything can come after $\rightarrow\uparrow\uparrow$. 
After $\rightarrow\uparrow$ and $\rightarrow\rightarrow\uparrow$, only $\rightarrow\rightarrow\uparrow$ is allowed.

\begin{ex}
If  $\rightarrow\uparrow$ or $\rightarrow\rightarrow\uparrow$ is followed by $\rightarrow\uparrow$ and $\rightarrow\uparrow\uparrow$, a $3$-hook will appear.

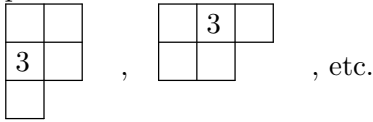
\unitlength 0.1in
\begin{picture}(16.00,6.00)(4.00,-10.00)
%
\special{pn 8}%
\special{pa 400 400}%
\special{pa 400 1000}%
\special{fp}%
\special{pa 400 1000}%
\special{pa 600 1000}%
\special{fp}%
\special{pa 600 1000}%
\special{pa 600 800}%
\special{fp}%
\special{pa 600 800}%
\special{pa 800 800}%
\special{fp}%
\special{pa 800 800}%
\special{pa 800 400}%
\special{fp}%
\special{pa 800 400}%
\special{pa 400 400}%
\special{fp}%
\special{pa 400 600}%
\special{pa 800 600}%
\special{fp}%
\special{pa 400 800}%
\special{pa 600 800}%
\special{fp}%
\special{pa 600 800}%
\special{pa 600 400}%
\special{fp}%
\put(4.5000,-7.5000){\makebox(0,0)[lb]{$3$}}%
\put(10.0000,-8.0000){\makebox(0,0)[lb]{, }}%
%
\special{pn 8}%
\special{pa 1200 400}%
\special{pa 1200 800}%
\special{fp}%
\special{pa 1200 800}%
\special{pa 1600 800}%
\special{fp}%
\special{pa 1600 800}%
\special{pa 1600 600}%
\special{fp}%
\special{pa 1800 600}%
\special{pa 1200 600}%
\special{fp}%
\special{pa 1200 400}%
\special{pa 1800 400}%
\special{fp}%
\special{pa 1800 400}%
\special{pa 1800 600}%
\special{fp}%
\special{pa 1600 600}%
\special{pa 1600 400}%
\special{fp}%
\special{pa 1400 400}%
\special{pa 1400 800}%
\special{fp}%
\put(14.5000,-5.5000){\makebox(0,0)[lb]{$3$}}%
\put(20.0000,-8.0000){\makebox(0,0)[lb]{, etc.}}%
\end{picture}%
\end{ex}

Therefore, the rim movement of the $3$-core is as follows. 
First $\rightarrow\uparrow\uparrow$ appears consecutively, then $\rightarrow\uparrow$ appears at most once, 
and then $\rightarrow\rightarrow\uparrow$ appears consecutively. Therefore, the generating function is
\[
\sum_{\lambda\in {\mathcal{C}}_{(3)}}{M(\lambda)}=\frac{1}{1-YX^2}(1+YX)\frac{1}{1-Y^2X}. 
\]
Similarly for $p=4$, 
\begin{eqnarray*}
&&\sum_{\lambda\in {\mathcal{C}}_{(4)}}{M(\lambda)}\\
&=&\frac{1}{1-YX^3}
\left(
\frac{1}{1-YX}+
(1+YX^2)\frac{1}{1-Y^2X^2}(1+Y^2X)
-1
\right)\\
&&\times\frac{1}{1-Y^3X}.
\end{eqnarray*}
Next is regarding $vhC_{(p)}$. 
The generating function of this set is difficult to construct as a normal generating function whose size appears in the exponent. 
It is easy if there is only one condition, vartical hook or horizontal hook. 
However, 
the generating function of the partition sequence of this set is straightforwad. 
The only condition is that each $X$ and $Y$ is not consecutive for $p$ times. 
\[
\displaystyle
\sum_
{\lambda\in vhC_{(p)}}{M(\lambda)}
=
\frac{1}{1-\sum_{1\leq a,b<p}{Y^aX^b}}. 
\]


\begin{thebibliography}{8}
\bibitem{en}
H. Enomoto, 
Mathematical theory of the Maya game (lecture by Mikio Sato), 
Problems of computer games and puzzles, RIMS Kokyuroku 98, 105-135, 1970.
\bibitem{ol}
J. B. Olsson,
Combinatorics and representations of finite groups, Vorlesungen aus dem Fachbereich Mathematik der Universit\"at GH Essen 20, 1993.
\bibitem{w}
C. P. Welter, The theory of a class of games on a sequence of squares, in terms of
advancing operation in a special group, Indag. Math. 16, 194-200, 1954.
\end{thebibliography}
\end{document}